\documentclass{article}
\usepackage{amsthm, amsfonts, amssymb,latexsym}

\title{Arcs in $\F_q^2$}
\newcommand{\F}{\mathbb{F}}
\newcommand{\E}{\mathbb{E}}

\renewcommand{\P}{\mathbb{P}}
\usepackage{amsmath}
\usepackage[a4paper, total={6in, 8in}]{geometry}
\newcommand{\cL}{\mathcal{L}}
\newcommand{\cH}{\mathcal{H}}

\newcommand{\cC}{\mathcal{C}}

\author{Oliver Roche-Newton and Audie Warren}

\newtheorem{lemma}{Lemma}

\newtheorem{theorem}{Theorem}

\usepackage{color}
\usepackage[dvipsnames]{xcolor}

\begin{document}
 \maketitle
 \begin{abstract}
 An arc is a subset of $\mathbb F_q^2$ which does not contain any collinear triples. Let $A(q,k)$ denote the number of arcs in $\mathbb F_q^2$ with cardinality $k$. This paper is primarily concerned with estimating the size of $A(q,k)$ when $k$ is relatively large, namely $k=q^t$ for some $t>0$. Trivial estimates tell us that
 \[
 {q \choose k}  \leq A(q,k) \leq {q^2 \choose k}.
 \]
 We show that the behaviour of $A(q,k)$ changes significantly close to $t=1/2$. Below this threshold an elementary argument is used to prove that the trivial upper bound above cannot be improved significantly.
 
 On the other hand, for $t \geq 1/2+\delta$, we use the theory of hypergraph containers to get an improved upper bound
 \[
 A(q,k) \leq {q^{2-t+2\delta} \choose k}.
 \]
 
 This technique is also used to give an upper bound for the size of the largest arc in a random subset of $\mathbb F_q^2$ which holds with high probability. For example, we prove that a $p$-random subset $Q \subset \mathbb F_q^2$ with $q^{-3/2}<p<q^{-1}$ contains an arc of size $\Omega(q^{1/2})$ with high probability. The result is optimal for this range of $p$.
 
 Finally, this optimal bound for arcs in random sets is used to prove a finite field analogue of a result of Balogh and Solymosi \cite{BS}, with a better exponent: there exists a subset $P \subset \mathbb F_q^2$ which does not contain any collinear quadruples, but with the property that for every $P' \subset P$ with $|P'| \geq |P|^{3/4+o(1)}$, $P'$ contains a collinear triple.

 \end{abstract}
 \section{Introduction} 
 
 \subsection{Basics}
 
Let $\F_q$ be the finite field of order $q = p^r$ for some prime $p$. An arc in $\F_q^2$ is a subset of $\F_q^2$ with no three points collinear. A $k$-arc is an arc with cardinality $k$. Define $A(q,k)$ to be the number of such $k$-arcs. In this paper, we are interested in estimating $A(q,k)$, particularly in the case when $k$ is relatively large with respect to $q$.

 Previous work on this problem has focused on what happens for small values of $k$, and exact formulas for $A(q,k)$ are known for $2 \leq k \leq 9$, see \cite{ISS} and the references within. For slightly larger values of $k$, Kaipa \cite{K} provided an upper bound for $A(q,k)$. However, the result of Kaipa is only effective when $k=O(\log q)$.
 
 Our focus is on bounding $A(q,k)$ in the case $k=q^t$, for some $t>0$. To provide some context, let us observe some trivial bounds for $A(q,k)$. Firstly, note that any subset of an arc is also an arc. Since the set
 \[ 
 C=\{(x,x^2): x \in \mathbb F_q\}
 \]
 is an arc of cardinality $q$, all of the
  $q \choose{k}$ subsets of $C$ of size $k$ are arcs, and hence  $A(q,k) \geq {q \choose k}$. 

A trivial upper bound for $A(q,k)$ is given by the number of subsets of $\mathbb F_q^2$ of cardinality $k$. To summarise, we have the following trivial bounds for $A(q,k)$:
\begin{equation} \label{trivcontext}
{q \choose k}  \leq A(q,k) \leq {q^2 \choose k} .
\end{equation}

\subsection{Counting $k$-arcs}

In this paper, we show that a threshold occurs at $t=1/2$ (recall that $k=q^t$), at which point the behaviour of $A(q,k)$ appears to change considerably. For $t \leq 1/2$, an elementary probabilistic argument gives a rather precise description of $A(q,k)$.

\begin{theorem} \label{thm:smallt} Suppose $k\leq \frac{q^{1/2}}{1 + \delta}$ for some $\delta > 0$. Then there exist an absolute constant $c>0$ and a constant $C = C(\delta)>0$ such that
 \[
 {q^2 \choose k} e^{\frac{-Ck^3}{q}} \leq A(q,k) \leq {q^2 \choose k} e^{\frac{-ck^3}{q}}.
 \]
 \end{theorem}
 
 Note that this value is rather close to the trivial upper bound given in \eqref{trivcontext}. For example,
 \[
 {q^2 \choose k} e^{\frac{-Ck^3}{q}} \geq {c'q^2 \choose k} 
 \]
 for some constant $c'>0$ (depending on $C$), provided that $q$ is sufficiently large with respect to $C$.

The argument leading to Theorem \ref{thm:smallt} fails for $t > 1/2$, and instead the machinery of hypergraph containers is used to give a much stronger upper bound. One of our main results is the following, which deals with this case.
 
 \begin{theorem} \label{thm:larget}
Let $\delta >0$, and suppose that $q$ is a sufficiently large (with respect to $\delta$) prime power. Then for all $t \geq \frac{1}{2}+\delta$
 \begin{equation} \label{Aboundlarge}
 A(q,k) \leq {2q^{2-t+3\delta/2} \choose k}.
 \end{equation}
 \end{theorem}
 
 \subsection{Two thresholds}
 
These results show that there is a sudden change in the behaviour of $A(q,k)$, with the exponent in the top part of the binomial coefficient suddenly dropping from almost $2$ to at most $3/2$. For much larger values of $t$ (i.e. for $t \rightarrow 1$) the upper bound given by Theorem \ref{thm:larget} approaches the trivial lower bound from \eqref{trivcontext}.
 
 It is possible that this transition is even more sharp, and it is even conceivable that $A(q,k)$ is close to the trivial lower bound from \eqref{trivcontext} for all $t>1/2$, but we were not able to prove or disprove this.

Furthermore, another sharp threshold is present in the statement of Theorem 1, occurring when $k \approx q^{1/3}$. The function $e^{\frac{-Ck^3}{q}}$ estimates the probability that a random set of points of size $k=q^t$ forms an arc. When $t<1/3$ this probability tends to $1$ as $q$ grows, but when $t>1/3$ it tends to zero. Some similar occurrences of such sharp probabilistic transitions in behaviour of combinatorial structures can be found in, for example, \cite{CG}, \cite{OPT}, and \cite{Sch}.

 \subsubsection*{Notation}
 
Throughout the paper, the standard notation
$\ll,\gg$ and respectively $O$ and $\Omega$ is applied to positive quantities in the usual way. That is, $X\gg Y$, $Y \ll X,$ $X=\Omega(Y)$ and $Y=O(X)$ all mean that $X\geq cY$, for some absolute constant $c>0$. If both $X \ll Y$ and $Y \ll X$ hold we write $X \approx Y$, or equivalently $X= \Theta(Y)$. All logarithms are in base $2$, unless stated otherwise.

\subsection{The largest arc contained in a random set}

Given $P \subset \mathbb F_q^2$, let $a(P)$ denote the size of the largest arc $P'$ such that $P' \subseteq P$. Let $Q$ be a random subset of $\mathbb{F}_{q}^{2}$ with the events $x \in Q$ being independent with probability $\mathbb{P}(x \in Q) = p$. We say that $Q$ is a \textit{$p$-random} set. The question of how large $a(Q)$ can be for a $p$-random set $Q$ is considered in this paper. We prove the following result.

\begin{theorem} \label{thm:random}
Suppose that $0<p<1$ and let $Q \subseteq \mathbb F_q^2$ be a $p$-random set. Let $\delta>0$.
\begin{itemize}
    \item If $p < 1/q$ then
    \[
    \lim_{q \rightarrow \infty} \mathbb P [a(Q) \leq  q^{\frac{1}{2}+2\delta}]= 1.
    \]
    \item If $1/q \leq p  \leq q^{-4\delta}$ then
     \[
    \lim_{q \rightarrow \infty} \mathbb P [a(Q) \leq  q^{1+2\delta}p^{1/2}]= 1.
    \]
\end{itemize}
\end{theorem}

In the range $q^{-3/2}<p<q^{-1}$, the first part of Theorem \ref{thm:random} is optimal up to the infinitesimal $\delta$. In this range, a simple argument can be used to show that, with high probability, $a(Q) \gg q^{1/2}$. The details of this argument will be given in Section \ref{sec:random}. For the even smaller range $p<q^{-3/2}$, this question becomes trivial. On the other hand, we think that it is likely that the second part of Theorem \ref{thm:random}, dealing with the case when $p$ is fairly large, is not optimal. See also Section \ref{sec:random} for further discussion.
 
 \subsection{Containers, supersaturation and the Balogh-Solymosi Theorem}
 
 The proofs of Theorems \ref{thm:larget} and \ref{thm:random} make use of hypergraph containers. The theory of hypergraph containers was developed independently by Balogh, Morris and Samotij \cite{BMS} and Saxton and Thomason \cite{ST}. We defer the full statement of the container theorem we use until Section \ref{sec:cont}. Roughly speaking, it says that if a hypergraph has a reasonably good edge distribution, we can obtain strong information about where the independent sets of the hypergraph may be found.
 
 This new method has led to several significant breakthroughs in combinatorics in recent years, most notably in the field of extremal graph theory. Of more relevance to this paper is the work of Balogh and Solymosi \cite{BS}, in which they prove the existence of point sets in $\mathbb R^2$ which do not contain collinear quadruples, but all large subsets contain a collinear triple.
 
 \begin{theorem}[\cite{BS}, Theorem 2.1] \label{thm:BS}
 For all $\delta >0$, there exists $n_0 \in \mathbb Z$ such that for all $n \geq n_0$ there exists a set $P \subset \mathbb R^2$ with $|P|=n$ with the following properties. $P$ does not contain any collinear quadruples, but $a(P) \leq n^{5/6+\delta}$.
 \end{theorem}
 
 A key step in \cite{BS} is the application of the container theorem to give detailed information about the subsets of $[n]^3$ which do not contain any collinear triples. It is an adaptation of this argument to the $\mathbb F_q^2$ case which is used in the proofs of Theorems \ref{thm:larget} and \ref{thm:random}.
 
 In most container applications, an important component is a ``supersaturation lemma". In the context of collinear triples, we have already seen an example of a set $P \subset \mathbb F_q^2$ of size $q$ which does not contain any collinear triples. On the other hand, a celebrated result of Segre \cite{Seg} implies that any set of $q+2$ points in $\mathbb F_q^2$ contains at least one collinear triple.
 
A supersaturation lemma in this setting is a statement which says that as we increase the cardinality of our set beyond the threshold at which we are guaranteed at least one collinear triple, we find \emph{many} such triples. The precise statement and its proof can be found in Section \ref{sec:supersat}.

Our supersaturation lemma in $\mathbb F_q^2$ is optimal up to constants, whereas the corresponding result over $[n]^3$ in \cite{BS} is expected not to be. This allows us to give the following analogue of Theorem \ref{thm:BS} in the $\mathbb F_q^2$ with better exponents.

\begin{theorem} \label{thm:3and4}
For all $\delta>0$ and for all sufficiently large (depending on $\delta$) prime powers $q$, there exists a set $P \subset \mathbb F_q^2$ with $|P| \gg q^{2/3}$ such that $P$ does not contain any collinear quadruples, but $a(P) \leq |P|^{3/4+\delta} $.
\end{theorem}

It is not clear whether the exponent $3/4$ is optimal in Theorem \ref{thm:3and4}, although we suspect that it is not. On the other hand, a closer look at the proof of Theorem \ref{thm:3and4} reveals that it can be restated to give the following optimal result.

\begin{theorem} \label{thm:3and4'}
For all $\delta>0$ and for all sufficiently large (depending on $\delta$) prime powers $q$, there exists a set $P \subset \mathbb F_q^2$ with $|P| \gg q^{2/3}$ such that
\begin{itemize}
    \item $P$ does not contain any collinear quadruples,
    \item $a(P) \leq |P|^{3/4+\delta} $,
    \item $P$ contains $O(|P|^{3/2})$ collinear triples.
\end{itemize}
\end{theorem}

A simple argument (see the forthcoming Lemma \ref{lem:simple}) implies that any $P \subset \mathbb F_q^2$ with $O(|P|^{3/2})$ collinear triples satisfies $a(P) \gg |P|^{3/4}$. Therefore, Theorem \ref{thm:3and4'} cannot be improved, except possibly for the infinitesimal $\delta$.

The third condition in Theorem \ref{thm:3and4'} arises from the fact that $P$ is (a modification of) a $p$-random set. Theorem \ref{thm:3and4'} highlights the fact that the exponent $3/4$ in Theorem \ref{thm:3and4} is the limit of what can be proved using the current approach. The same is true in the Euclidean setting; the best we can hope for in Theorem \ref{thm:BS} using the current approach is an exponent $3/4$.

\subsection{A near-optimal bound for a variant of the Balogh-Solymosi Theorem}

Given two integers $m < l$, define
\[
f_{m,l}(n):= \min_{P \subset \mathbb F_q^2 : |P|=n, P \text{ does not contain any collinear $l$-tuples}} a(P).
\]
Theorem \ref{thm:3and4} says that $f_{3,4}(n) \leq n^{3/4 + \delta}$.

As we mentioned in the previous section, the optimal exponent for $f_{3,4}(n)$ is unknown. A simple greedy algorithm argument shows that any set of $n$ points with no collinear quadruples contains a subset of size $\Omega(n^{1/2})$ with no collinear triples, that is, $f_{3,4}(n) \gg n^{1/2}$. This can also be deduced from the forthcoming Lemma \ref{lem:simple}. A small improvement to this bound in the Euclidean setting was obtained by F\"{u}redi \cite{F}, using a deep graph theoretical result of Koml\'{o}s, Pintz and Szemer\'{e}di \cite{KPS}.

The argument used in the proof of Theorem \ref{thm:3and4} works more effectively to give near-optimal bounds for $f_{3,l}(n)$ as $l$ increases.

\begin{theorem} \label{thm:3andk}
Let $l \geq 4$ be an integer. For all $\delta>0$ and for all sufficiently large (depending on $\delta$) prime powers $q$, there exists a set $P \subset \mathbb F_q^2$ with $|P| \gg q^{2 - \frac{l}{l-1}}$ having the following properties:
\begin{itemize}
    \item $P$ does not contain any collinear $l$-tuples,
    \item $a(P) \leq |P|^{\left(\frac{l-1}{l-2}\right)\left(\frac{1}{2}+\delta\right)} $,
    \item $P$ contains $O(|P|^{\frac{2l-5}{l-2}})$ collinear triples.
\end{itemize}
\end{theorem}

Simply discarding the extra information contained in the third bullet point, note that Theorem \ref{thm:3andk} implies that
\begin{equation} \label{f3k}
f_{3,l}(n) \leq n^{\left(\frac{l-1}{l-2}\right)\left(\frac{1}{2}+\delta\right)}
\end{equation}
for all $\delta>0$. Taking $l$ to be sufficiently large, the exponent in \eqref{f3k} gets arbitrarily close to $1/2$.

To find a lower bound for $f_{3,l}$, the greedy algorithm argument, or again Lemma \ref{lem:simple},  can be used to show that any set $P \subset \mathbb F_q^2$ of size $n$ with no collinear $k$-tuples contains an arc $P' \subset P$ with $|P'| \gg n^{1/2}/k^{1/2}$. That is,
\[
f_{3,k} \gg \frac{n^{1/2}}{k^{1/2}}.
\]

Once again, Lemma \ref{lem:simple} implies that any $P \subset \mathbb F_q^2$ with $O(|P|^{\frac{2l-5}{l-2}})$ collinear triples satisfies $a(P) \gg |P|^{\frac{l-1}{2(l-2)}}$, and so Theorem \ref{thm:3andk} is tight, except possibly for the infinitesimal $\delta$.

Note that Theorem \ref{thm:3andk} implies Theorem \ref{thm:3and4'}, and thus also Theorem \ref{thm:3and4}. The proof of Theorem \ref{thm:3andk} follows from an application of Theorem \ref{thm:random} with $p$ approximately equal to $q^{\frac{-l}{l-1}}$. This value of $p$ lies in the range for which Theorem \ref{thm:random} is optimal.
 
\subsection{Counting MDS codes}

A linear code $C$ of dimension $k$ and length $n$ over $\mathbb F_q$ is a linear subspace of dimension $k$ in $\mathbb F_q^n$. The distance $d(c_1,c_2)$ between two elements of $c_1,c_2 \in C$ is the number of positions in which they differ. The minimum distance of $C$ is the minimum value of $d(c_1,c_2)$ over all distinct $c_1,c_2 \in C$. A fundamental idea in coding theory is to construct codes with large minimum distance, since these codes have the greatest error detecting and correcting capabilities. The Singleton bound implies that the minimum distance of any linear code is at most $n-k+1$, and \textit{maximum distance separable} (MDS) codes are those which attain this bound. 

Let $B(q,n)$ denote the number of $n$-arcs in the projective plane $\P(\F_q^2)$. A connection between the number of projective arcs and MDS codes was established in \cite{ISS}. Lemmas 1 and 2 therein imply that the number of MDS codes of dimension $3$ and length $n$ over $\mathbb F_q$ is
\begin{equation} \label{MDS}
\frac{n!(q-1)^{n-2}}{q^3(q^2+q+1)(q+1)}B(q,n).
\end{equation}

The arguments in this paper can also be framed in the projective setting to get similar bounds for $B(q,k)$. These modified versions of Theorems 1 and 2, combined with \eqref{MDS}, give bounds for the number of $[n,3]_q$ MDS codes. This is not the first time that containers have given applications in coding theory; Balogh, Treglown and Zsolt Wagner used the technique to count the number of $t$ error correcting codes in \cite{BTZ}.

\subsection{The structure of the rest of this paper}

In section \ref{sec:cont} we give the necessary background and then the statement of the container theorem that will be used, as well as proving our supersaturation lemma. We then use these two tools to prove Lemma \ref{lem:work}; this is our container lemma specified to the case of arcs, giving detailed quantitative information about where arcs in $\mathbb F_q^2$ may be found. In section \ref{sec:proofs}, we use Lemma \ref{lem:work} to prove what we consider to be the main results of this paper, namely Theorems \ref{thm:larget}, \ref{thm:random} and \ref{thm:3andk}. Recall that Theorem \ref{thm:3andk} implies Theorems \ref{thm:3and4} and \ref{thm:3and4'}. In section \ref{sec:smallk} we deal with the problem of bounding $A(q,k)$ when $k$ is smaller, namely when $k \leq \sqrt q$, proving Theorem \ref{thm:smallt}.

\section{Containers and supersaturation} 


\subsection{Statement of the container theorem} \label{sec:cont}
Before stating the container theorem, we introduce some related quantities. For an $r-$uniform hypergraph $\cH=(V,E)$ and $v \in V$, $d(v)$ denotes the degree of $v$. Let $d(\cH)$ denote the average degree of $\cH$, so
\[
d(\cH)= \frac{1}{|V|}\sum_{v \in V} d(v)=  \frac{r|E|}{|V|}.
\]
We can also define the \emph{co-degree} for a subset $S \subseteq \cH$ of vertices as
$$d(S) = \{ e \in E(\cH) : S \subseteq e \}.$$
Using this definition we define the \emph{k-maximum co-degree} $\Delta_k$ as
$$\Delta_k(\cH) = \max_{\substack{S \subseteq \cH \\ |S| = k}}d(S).$$
Since edges are subsets of size $r$, we see that $\Delta_r(\cH) = 1$ (or $\cH$ is empty), and $\Delta_k(\cH) = 0$ for all $k > r$. For any $V' \subset V$, $\cH[V']$ denotes the subgraph induced by $V'$.

We now state the container theorem we need, which is Corollary 3.6 in \cite{ST}.

\begin{theorem} \label{container}
Let $\cH=(V,E)$ be an $r-$uniform hypergraph on $n$ vertices, and let $\epsilon, \tau \in (0,1/2)$. Suppose that
\begin{equation} \label{cond1} 2^{{r \choose 2} - 1} \sum_{j=2}^r\frac{\Delta_j(\cH)}{d(\cH) \cdot \tau^{j-1} \cdot 2^{{j-1 \choose 2}}} \leq \frac{\epsilon}{12 r!}\end{equation}
and
\begin{equation} \label{cond2}
\tau < \frac{1}{200 r \cdot r!}.
\end{equation}
Then there exists a set $\cC$ of subsets of $V$ such that
\begin{enumerate}
    \item if $A \subset V$ is an independent set then there exists $C \in \cC$ such that $A \subset C$;
    \item $|E(\cH[C])| \leq \epsilon |E(\cH)|$ for all $C \in \cC$;
    \item $\log |C| \ll_r  n  \tau \log(\frac{1}{\epsilon})  \log(\frac{1}{\tau})$.  
\end{enumerate}
\end{theorem}
The set $\cC$ above is referred to as the set of \emph{containers,} and a set $C \in \cC$ is itself a container. 

\subsection{Supersaturation lemma} \label{sec:supersat}

The second point in the statement of Theorem \ref{container} above says that the number of edges in $\cH(C)$ must be smaller than the total number of edges in $\cH$. As in most applications of containers, we aim for this to ensure that the number of \emph{vertices} in the containers is small. This connection is given by the following supersaturation lemma.
\begin{lemma}\label{super}
Let $P \subset \F_q^2$ with $|P| \geq  4q$. Then the number of collinear triples defined by $P$ is $\Omega\left (\frac{|P|^3}{q}\right)$.
\end{lemma}
\begin{proof}
Let $\mathcal L$ denote the set of all lines in $\mathbb F_q^2$. Let 
\[
\cL':=\{l \in \cL : |l \cap P| \geq 3\}.
\]
Then
$$ (q+1)|P| = \sum_{l \in \cL} |l \cap P|=\sum_{l \in \cL'} |l \cap P|+\sum_{l \in \cL \setminus \cL'} |l \cap P| \leq \sum_{l \in \cL'} |l \cap P|+2q(q+1).$$
Since $|P| \geq 4q$, it follows that 
$$ \frac{(q+1)|P|}{2} \leq \sum_{l \in \cL'} |l \cap P| .$$
Applying the H\"{o}lder inequality gives
$$ q |P| \ll (q^2 + q)^{2/3}\left(\sum_{l \in \cL'} |l \cap P|^3 \right)^{1/3} \ll q^{4/3}\left(\sum_{l \in \cL'} |l \cap P|^3 \right)^{1/3} \ll q^{4/3}\left(\sum_{l \in \cL'} {|l \cap P| \choose 3} \right)^{1/3} ,$$
where the last inequality uses the fact that $|l \cap P| \geq 3$. A rearrangement gives
\[
\# \text{collinear triples in $P$}= \sum_{l \in \cL'} {|l \cap P| \choose 3} \gg \frac{|P|^3}{q}.
\]
\end{proof}

Lemma \ref{super} is optimal, up to the implied constant. To see this, first observe that $\mathbb F_q^2$ contains $\Omega(q^5)$ collinear triples; there are $q^2+q$ lines, and each line contains ${q \choose 3} \gg q^3$ collinear triples. Now define a random set $P \subset \mathbb F_q^2$, where each element $x \in \mathbb F_q^2$ belongs to $P$ with probability $p=q^{-s}$. With high probability we have $|P| =\Theta (q^{2-s})$, and the number of collinear triples in $P$ is $\Theta(q^{5-3s})$, since there are $q^5$ collinear triples to begin with, and each one survives the random selection process with probability $q^{-3s}$. Furthermore, the condition that $s<1$ is necessary, since we already saw a simple algebraic construction of an arc with cardinality $q$. The corresponding supersaturation result for the grid $[n]^3$ was given in \cite[Lemma 4.2]{BS}. It is tentatively believed that this result is not optimal. Any improvement would result in an improved exponent in Theorem \ref{thm:BS}.

\subsection{Application of the container theorem}

The hard work is done in the following lemma, which will be used to prove most of the results in this paper.

\begin{lemma} \label{lem:work}
Let $\delta>0$ and suppose that $q$ is a sufficiently large (with respect to $\delta$) prime power. Let $1/2+2\delta \leq t \leq 1$. Then there exists a family $\mathcal C$ of subsets of $\mathbb F_q^2$ such that
\begin{itemize}
    \item $|\mathcal C| \leq 2^{c(\delta)q^{t-\delta}(\log q)^2}$,
    \item For all $C \in \mathcal C$, $|C| \leq q^{2-t+3\delta}$,
    \item For every arc $P \subset \mathbb F_q^2$, there exists $C \in \mathcal C$ such that $P \subset C$.
\end{itemize}

\end{lemma}

\begin{proof}

Define a $3-$uniform hypergraph $\cH$ with vertices corresponding to points in $\F_q^2$, with three points forming a (hyper)edge if they are collinear. We now employ an idea used in \cite{BS}; we will iteratively apply Theorem \ref{container} to subsets of $\mathbb F_q^2$. We begin by applying it to the graph $\cH$. As a result, we obtain a set $\mathcal C_1$ of containers. We iterate by considering each $A \in \mathcal C_1$. If $A$ is not small enough, then we apply Theorem \ref{container} to the graph $\cH[A]$ to get a family of containers $\mathcal C_A$. If $A$ is sufficiently small then we put this $A$ into a final set $\mathcal C$ of containers (or to put it another way, we write $\mathcal C_A=A$). 

Repeating this for all $A \in \mathcal C_1$ we obtain a new set of containers
\[
\mathcal C_2 = \bigcup_{A \in  \mathcal C_1} \mathcal C_A.
\]
Note that $\mathcal C_2$ is a container set for $\cH$. Indeed, suppose that $X$ is an independent set in $\cH$. Then there is some $A \in \mathcal C_1$ such that $X \subset A$. Also, $X$ is an independent set in the hypergraph $\cH[A]$, which implies that $X \subset A'$ for some $A' \in \mathcal C_A \subset \mathcal C_2$.

We then repeat this process, defining
\[
\mathcal C_i = \bigcup_{A \in  \mathcal C_{i-1}} \mathcal C_A.
\]
By applying Lemma \ref{super} and choosing the values of $\tau$ and $\epsilon$ appropriately, we can ensure that after relatively few steps we have all of the elements of $\mathcal C_m$ sufficiently small. We then declare $\mathcal C= \mathcal C_m$. It turns out that, because of $m$ being reasonably small, $|\mathcal C|$ is also fairly small.

Now we give more precise details of how to run this argument. Let $A \in \mathcal C_j$, with $j \leq m$, and write $|A|=q^{2-s}$. If $s \geq t-3\delta$ then do nothing. Otherwise, we will apply Theorem \ref{container} to $\cH[A]$.

 Since $\cH$ is $3-$uniform we have $\Delta_3(\cH) = 1$. Furthermore, given a pair of points in the plane, the number of points which are collinear with the given pair is $q-2$, and thus $\Delta_2(\cH) = q-2$. The subgraph $\cH[A]$ therefore satisfies the bounds $\Delta_3(\cH[A]) \leq 1$ and $\Delta_2(\cH[A]) \leq q-2$. We also have the lower bound on the average degree of $\cH[A]$, which follows from Lemma \ref{super}:
$$d(\cH[A]) = \frac{1}{q^{2-s}}\sum_{v \in \cH[A]}d(v) = \frac{3E(\cH[A])}{q^{2-s}} \geq cq^{3-2s}. $$
Here $c>0$ is the absolute constant coming from Lemma \ref{super}.
We consider the two conditions \eqref{cond1} and \eqref{cond2} present in Theorem \ref{container}. Applying the above bounds, the condition \eqref{cond1} holds if
\begin{equation}\label{condition}\frac{1}{cq^{2-2s} \tau} + \frac{ 1}{2cq^{3-2s} \tau^2} \leq \frac{\epsilon}{288}.\end{equation}
Let $\delta >0$ be an absolute constant. We make the choices
$$\tau =  q^{s+t-2-\delta}, \qquad \epsilon =  q^{- \delta}.$$
Plugging these values into \eqref{condition}, this becomes
\[
\frac{1}{cq^{t-s - \delta}} + \frac{ 1}{2cq^{2t-1 - 2\delta}} \leq \frac{q^{-\delta}}{288}.
\] 
It would suffice that
\[
q^{t-s-2\delta} \geq c^{-1}1000,\,\,\,\,\,\ \text{and} \,\,\,\,\,\, q^{2t-1-3\delta}\geq c^{-1}1000.
\]
Since $s < t -3\delta$, the first of these inequalities holds for $q$ sufficiently large (depending on $\delta$). Similarly, since $t \geq \frac{1}{2}+2\delta$, the second inequality holds for $q$ sufficiently large. Thus \eqref{condition} holds, and so does condition \eqref{cond1} of Theorem \ref{container}.

Condition \eqref{cond2} in this instance becomes
\[
q^{s+t-2 -\delta}<\frac{1}{3600}.
\]
Since $s+t-2-2\delta <2t-2-2\delta<-2\delta$, this is satisfied as long as $q$ is sufficiently large with respect to $\delta$. The same is true for the condition $\epsilon < 1/2$. With these choices of $\tau$ and $\epsilon$, Theorem \ref{container} can be legitimately applied. We obtain a set of containers $\mathcal C_A$ with
\[
|\mathcal C_A| \ll 2^{q^{t-\delta}(\log q)^2}.
\]
We also know that for each $B \in \mathcal C_A$,
\[|E(\cH[B])| \leq \epsilon |E(\cH[A])| = q^{-\delta}|E(\cH[A])|.
\]
Therefore, at the $i$th level of this iterative procedure a container $B \in \mathcal C_i$ satisfies
\[ |E(\cH[B])| \leq q^{5-i\delta}.
\]
Therefore, after $m(\delta)$ steps of this iteration, we can ensure that all of the containers $B$ at this level satisfy
\begin{equation} \label{nearly}
 |E(\cH[B])| \leq c'q^{5-3t},
\end{equation}
where $c'>0$ is chosen to be sufficiently small so that Lemma \ref{super} implies that $|B| \leq q^{2-t}$. That is, after $m(\delta)$ steps, this process will terminate. Choosing $m=\frac{4t}{\delta}$ will amply suffice, assuming once again that $q$ is sufficiently large.

 After the final iteration, we take the union of all containers we have found, and call this final set $\cC$. We have 
 \[
 |\cC| \ll 2^{\frac{4t}{\delta}q^{t -\delta}(\log q)^2}.
 \]
Each $C \in \mathcal C$ has size at most $q^{2-t+3\delta}$. Since independent sets in the hypergraph $\cH$ correspond precisely with arcs in $\mathbb F_q^2$, the set $\cC$ has the properties claimed in the statement of Lemma \ref{lem:work}.

\end{proof}

\section{Proofs of the main results} \label{sec:proofs}

\subsection{Proof of Theorem \ref{thm:larget}}

We can now bound the number of $k-$arcs. 

\begin{proof}[Proof of Theorem \ref{thm:larget}]

Fix $\delta>0$ and recall that $k = q^t$ for $t \geq 1/2+\delta$. Apply Lemma \ref{lem:work} with $\delta/2$. Then
\begin{equation}A(q,k)\leq 2^{c(\delta)q^{t - \delta/2}(\log q)^2}{q^{2-t+3\delta/2} \choose k}
\leq 2^k {q^{2-t+3\delta/2} \choose k} \leq  {2q^{2-t+3\delta/2} \choose k}.
\end{equation}
The second inequality above is valid provided that $q$ is sufficiently large with respect to $\delta$, while the final inequality is an instance of the bound $2^k \binom{a}{k} \leq  \binom{2a}{k}$. Since $\delta>0$ was arbitrary, the proof is complete.

\end{proof}

\subsection{Proof of Theorem \ref{thm:random} and some remarks on its optimality} \label{sec:random}

We will use Lemma \ref{lem:work} to prove Theorem \ref{thm:random}. Restating Theorem \ref{thm:random} in its contrapositive form, our task is to show that for a $p$-random set $Q \subset \mathbb F_q^2$ and any $\delta>0$,
\begin{enumerate}
    \item If $p < 1/q$ then
    \[
    \lim_{q \rightarrow \infty} \mathbb P [a(Q) \geq  q^{\frac{1}{2}+2\delta}]= 0.
    \]
    \item If $1/q \leq p  \leq q^{-4\delta}$ then
     \[
    \lim_{q \rightarrow \infty} \mathbb P [a(Q) \geq  q^{1+2\delta}p^{1/2}]= 0.
    \]
    \end{enumerate}

\begin{proof}[Proof of Theorem \ref{thm:random}]
We begin with part 1 of the theorem. Let $0<p<1/q$.
Apply Lemma \ref{lem:work} with this $\delta$ and $t=1/2+2\delta$, giving a set of containers $\cC$. For a parameter $m$ to be specified later, the probability that $Q$ contains an arc of size at least $m$ is upper bounded by
\begin{equation} \label{probbound}
 |\mathcal C| \binom{q^{3/2+\delta}}{m} p^m.
\end{equation}
This is because an arc of size $m$ must be contained in some $C \in \mathcal C$, and each subset of size $m$ belongs to the random subset $Q$ with probability $p^m$. Every $C \in \mathcal C$ has size
\[
|C| \leq q^{3/2+\delta},
\]
and so the number of possible candidates for an arc of size $m$ is at most
\[
 |\mathcal C| \binom{q^{3/2+\delta}}{m}.
\]
An application of the union bound then gives \eqref{probbound}.

Set $m=q^{1/2+2\delta}$. Then, provided that $q$ is sufficiently large, 
the probability that $Q$ contains an arc of size $m$ is at most
\begin{align} \label{long} \nonumber
2^{c(\delta)q^{t-\delta}(\log q)^2} \binom{q^{3/2+\delta}}{m} p^m \leq 2^{c(\delta)q^{1/2+\delta}(\log q)^2}\left( \frac{eq^{\frac{3}{2}+\delta}p}{q^{\frac{1}{2}+2\delta}}  \right)^{q^{1/2+2\delta}}
 & \leq  \left( \frac{2eq^{\frac{3}{2}+\delta}p}{q^{\frac{1}{2}+2\delta}}  \right)^{q^{1/2+2\delta}}
 \\&= \left( 2epq^{1-\delta} \right)^{q^{1/2+2\delta}}.
\end{align}
In the first inequality above, we have used the bound ${ a \choose b} \leq \left( \frac{ea}{b} \right)^b$.

Finally, since $p <1 / q$, it follows from \eqref{long} that
\[
\lim_{q \rightarrow \infty} \mathbb P [Q \text{ contains an arc of size at least } q^{1/2 + \delta}] = 0,
\]
as required.

We now turn to the proof of part 2 of the theorem. The proof is similar. Suppose that
\begin{equation} \label{assume}
q^{-1} \leq p \leq q^{-4\delta}.
\end{equation}
Apply Lemma \ref{lem:work} with $t=1+2\delta+\frac{1}{2} \log_q p$. Note that the assumption \eqref{assume} implies that $1/2+2\delta \leq t \leq 1$, and so the application of Lemma \ref{lem:work} is valid. Then, provided that $q$ is sufficiently large, the probability that a $p$-random set $Q$ contains an arc of size at least $m=q^{1+2\delta}p^{1/2}$ is at most
\begin{equation} \label{long2} 
2^{c(\delta)q^{t-\delta}(\log q)^2} \binom{q^{1+\delta}p^{-1/2}}{m} p^m \leq 2^{c(\delta)q^{1+\delta}p^{1/2}(\log q)^2}\left( \frac{eq^{1+\delta}}{q^{1+2\delta}}  \right)^{q^{1+2\delta}p^{1/2}}
\leq  \left( 2eq^{-\delta}  \right)^{q^{1/2+2\delta}}.
\end{equation}
It therefore follows that
\[
\lim_{q \rightarrow \infty} \mathbb P [Q \text{ contains an arc of size at least } q^{1 + 2\delta}p^{1/2}] = 0,
\]
as required.

\end{proof}


We conclude this subsection with a discussion of when Theorem \ref{thm:random} is optimal, and when it is likely not. Given $P \subset \mathbb F_q^2$, let $T(P)$ denote the number of collinear triples in $P$. We will need the following simple lemma.

\begin{lemma} \label{lem:simple}
Let $P \subset \mathbb F_q^2$, and suppose that $T(P) \geq |P| /2$. Then $a(P) \gg |P|^{3/2}/T(P)^{1/2}$.
\end{lemma}

\begin{proof} Let $P'$ be a $p$-random subset of $P$, with $0<p<1$ to be determined. Note that
\[
\mathbb E[|P'|]=p|P|,\,\,\,\, \mathbb E[T(P')]=T(P)p^3.
\]
Therefore, by linearity of expectation,
\[
\mathbb E[ |P'| -T(P')]= p|P| - T(P) p^3.
\]
Set $p=\frac{|P|^{1/2}}{\sqrt 2 T(P)^{1/2}}$. Note that the assumption of the lemma implies that $p \leq 1$. Then
\[
\mathbb E[ |P'| -T(P')]= \frac{|P|^{3/2}}{2 \sqrt 2T(P)^{1/2}}.
\]
Therefore, there exists a set $P' \subset P$ such that 
\begin{equation} \label{toprune}|P'|-T(P') \geq \frac{|P|^{3/2}}{2 \sqrt 2T(P)^{1/2}}.
\end{equation} 
Now prune $P'$ to get $P'' \subset P'$ with no collinear triples. For every collinear triple in $P'$ remove one element. After deleting at most $T(P')$ elements, we obtain a set $P''$ with no collinear triples. Then \eqref{toprune} implies that $|P''| \geq \frac{|P|^{3/2}}{2 \sqrt 2T(P)^{1/2}}$, as required.
\end{proof}

Lemma \ref{lem:simple} can be used to obtain the following lower bounds for $a(Q)$ for a $p$-random set. The purpose of this result is to illustrate that Theorem \ref{thm:random} is optimal in the range $q^{-3/2}<p<q^{-1}$.

\begin{lemma}
Let $Q$ be a $p$-random subset of $\mathbb F_q^2$ with $q^{-3/2}<p<1$. Let $\delta>0$ and assume that $q$ is sufficiently large with respect to $\delta$. Then, 
with probability at least $1-\delta$
\begin{equation} \label{p1}
|a(Q)| \gg_{\delta}  q^{1/2}.
\end{equation}

\end{lemma}

\begin{proof}
First we will prove \eqref{p1}. The expected number of collinear triples in a $p$-random set $Q$ is $p^3(q^2 + q){q \choose 3}$, and so by an application of Markov's inequality, we have
$$\P\left(T(Q) \geq \frac{2p^3(q^2 + q){q \choose 3}}{\delta} \right) \leq \delta/2.$$
We now calculate the probability that $|Q| \geq \frac{pq^2}{2}$:
\begin{align*}
    \P\left(|Q| \geq \frac{pq^2}{2}\right) & \geq \P\left(||Q| - \E(|Q|)| \leq \frac{pq^2}{2}\right) \\
    & = 1 - \P\left(\left||Q| - \E(|Q|)\right| \geq \frac{pq^2}{2}\right). 
\end{align*}
We use Chebyshev's inequality, which calls for calculating the variance of $|Q|$. We have
$$Var(|Q|) = \sum_{x \in \F_q^2} Var(X_x),$$
where $X_x$ is the (Boolean) event that the point $x$ is selected for $Q$. The events $X_x$ and $X_{x'}$ are independent, and so the covariance does not appear in the above equation. We have
$$Var(X_x) = \E(X_x^2) - \E(X_x)^2 = p - p^2.$$
We therefore have
$$Var(|Q|) =q^2p(1-p)$$
and so Chebyshev's inequality states that for all $\lambda>0$,
$$\P\left(||Q| - pq^2| \geq \lambda qp^{1/2}(1-p)^{1/2}\right) \leq \frac{1}{\lambda^2}.$$
Applying this inequality with $\lambda = \frac{p^{1/2}q}{2(1-p)^{1/2}}$ gives
$$\P\left(||Q| - pq^2| \geq\frac{q^2p}{2}\right) \leq \frac{4(1-p)}{pq^2} \leq \frac{4}{q^{1/2}}.$$
Finally, we find that
$$ \P\left(|Q| \geq \frac{pq^2}{2}\right) \geq 1 - \frac{4}{q^{1/2}} \geq 1-\delta/2,$$
where the last inequality holds for $q$ sufficiently large (with respect to $\delta$). Therefore, with probability at least $1 - \delta$, both of the bounds
\begin{equation} \label{twothings}
|Q| \geq \frac{pq^2}{2},\,\,\,\,\,\,\,T(Q) \leq \delta^{-1}p^3(q^2+q)\binom{q}{3}
\end{equation}
hold. Assume $Q$ is a set where both of these bounds occur. The proof is now split into two cases.

\textbf{Case 1} - Suppose that $T(Q) \geq |Q|/2$. Then apply Lemma \ref{lem:simple} to this $Q$. It follows from \eqref{twothings} that $a(Q) \gg (pq^2)^{3/2}/(\delta^{-1}p^3q^5)^{1/2}=\delta^{1/2}q^{1/2}$. 

\textbf{Case 2} - Suppose that $T(Q) < |Q| / 2$. In this case, we can prune $Q$ to get a subset $Q'$ with $|Q'| \geq |Q|/2$ containing no collinear triples. Indeed, for each collinear triple in $Q$, simply remove one element to destroy the triple, until no more remain. Thus we have $a(Q) \gg |Q| \gg pq^2 \geq q^{1/2}$. The last inequality uses the assumption that $p \geq q^{-3/2}$. This completes the proof of \eqref{p1}.

\end{proof}

For larger values of $p$, a better estimate for $a(Q)$ can be obtained by a different argument. Let $C=\{(x,x^2): x \in \mathbb F_q\}$. This is an arc with cardinality $q$. Define $Q':= Q \cap C$. By Chebyshev's Inequality as applied above, we have, with high probability, both $|Q| \ll pq^2$ and $|Q'| \gg qp$. Since $Q'$ is an arc in $Q$, we then have, with high probability
\[
a(Q) \geq |Q'| \gg qp  \gg |Q|q^{-1}.
\]

For $p<q^{-3/2}$ the situation is even simpler, since with high probability the number of collinear triples in $Q$ is significantly smaller than $|Q|$, and so we can prune $Q$ to get a large subset $Q'$ with no collinear triples and $|Q'| \gg |Q|$. That is, $a(Q) \gg |Q|$.

\subsection{Proof of Theorem \ref{thm:3andk}}

We now deduce Theorem \ref{thm:3andk} from Theorem \ref{thm:random}.

\begin{proof}[Proof of Theorem \ref{thm:3andk}]

 Let $p=\frac{q^{\frac{-l}{l-1}}}{100}$ and construct a $p$-random subset $Q \subset \mathbb F_q^2$, each element of $x \in \mathbb F_q^2$ belonging to $Q$ with probability $p$. Let $T_l(Q)$ denote the number of collinear $l$-tuples in $Q$. We will show that, with positive probability, all of the following statements are simultaneously true:
 \begin{enumerate}
 \item $|Q| \geq \frac{1}{2} \mathbb E[|Q|]=\frac{1}{200}q^{\frac{l-2}{l-1}}$
     \item $T_l(Q) \leq \frac{1}{2}|Q|$,
     \item $T(Q) \leq q^{\frac{2l-5}{l-1}}$,
     \item $a(Q) \leq q^{\frac{1}{2}+\delta}$.
 \end{enumerate}
 
 Suppose that, with positive probability, all of these statements hold, and in particular there is some set $Q$ with all of these properties. Fix this set $Q$. We prune $Q$ to find a subset $|P| \gg q^{\frac{l-2}{l-1}}$ with no collinear $l$-tuples. That is, for every collinear $l$-tuple in $Q$, remove one element to destroy the $l$-tuple. Because of point 2 above, at the end of this process, we have deleted at most half of the points and have no collinear $l$-tuples. Furthermore, the number of collinear triples has not increased, so $T(P) \leq q^{\frac{2l-5}{l-1}}$. Also, since $Q$ does not contain an arc of size $q^{1/2 +\delta}$, $P$ also does not, so $a(P) \leq q^{\frac{1}{2}+\delta}$. Therefore, the set $P$ has all of the properties claimed in the statement of Theorem \ref{thm:3andk}, and the proof is complete.
 
 It remains to prove that the four statements above hold simultaneously with positive probability. Let $\mathcal E_1$ be the event that $|Q| \leq \frac{1}{2}\mathbb E[|Q|]$. Note that $Var[|Q|] \leq \mathbb E[|Q|]$, and so by Chebyshev's Inequality
 \[
 \mathbb P[\mathcal E_1] \leq \mathbb P\left [ | |Q| - \mathbb E[|Q|] | \geq\frac{1}{2} \mathbb  E[|Q|] \right] \leq \frac{4}{\mathbb E[|Q|]} \leq \frac{1}{10},
 \]
 where the last inequality holds for $q$ sufficiently large.
 
 Let $\mathcal E_2$ be the event that $T_l(Q) \geq \frac{1}{4} \mathbb E[|Q|]$. By Markov's inequality
 \begin{equation} \label{mark1}
 \mathbb P [\mathcal E_2] \leq 4\frac{ \mathbb E [T_l(Q)]}{\mathbb E[|Q|]}.
 \end{equation}

The expected number of collinear $l$-tuples in $Q$ is $p^l\binom{q}{l}(q^2+q)$, since there are $\binom{q}{l}(q^2+q)$ collinear $l$-tuples in $\mathbb F_q^2$ and each $l$-tuple survives the random selection process with probability $p^l$. The expected size of $Q$ is $pq^2$. Therefore, \eqref{mark1} gives
\[
 \mathbb P [\mathcal E_2] \leq 4 \frac{p^lq^{l+2}}{pq^2}=4(100)^{1-l} <1/10.
\]

Similarly, let $\mathcal E_3$ be the event that $T(Q) \geq 10 \mathbb E[T(Q)]$. Markov's inequality implies that $\mathbb P [\mathcal E_3] \leq \frac{1}{10}$.

Let $\mathcal E_4$ be the event that $a(Q) \geq q^{1/2+\delta}$. An application of Theorem \ref{thm:random} implies that this probability tends to zero as $q$ goes to infinity. In particular, by taking $q$ sufficiently large, we certainly have $\mathbb P[ \mathcal E_4] \leq \frac{1}{10}$.

Since each of these four events occur with probability at most $1/10$, the probability that none of the four events hold is positive. Finally, we will prove that $\neg ( \mathcal E_1 \cup \mathcal E_2 \cup \mathcal E_3 \cup \mathcal E_4)$ implies that each of the the four required properties at the beginning of the proof hold.
\begin{enumerate}
    \item Since $\mathcal E_1$ does not hold, we have $|Q| \geq \frac{1}{2}\mathbb E[|Q|]=\frac{1}{200}q^{\frac{l-2}{l-1}}$.
    \item $\neg \mathcal E_2$ and $\neg \mathcal E_1$ together imply that $T_l(Q) \leq \frac{1}{4}\mathbb E[|Q|] \leq \frac{1}{2} |Q|$.
    \item Since $\mathcal E_3$ does not hold, we have $T(Q) \leq 10 \mathbb E[ T(Q)] \leq 10 p^3q^5 \leq  q^{\frac{2l-5}{l-1}}$.
    \item Since $\mathcal E_4$ does not hold, we immediately have $a(Q) \leq q^{\frac{1}{2}+\delta}$.
\end{enumerate}





\end{proof}

 \section{Counting smaller arcs} \label{sec:smallk}
 
 Let $L(P)$ be the set of lines containing at least two points from $P$. Define the point set
 \[
 \mathcal L_P:= \bigcup_{\ell \in L(P)} \ell.
 \]
 We begin with a simple lemma.
 
 \begin{lemma} \label{lem:easy} Let $P \subset \mathbb F_q^2$ be an arc with $|P|=k\leq \sqrt q$. Then
 \[
 \frac{q}{2}\binom{k}{2} \leq |\mathcal L_P|\leq q \binom{k}{2}.
 \]
 
 \end{lemma}
 
 \begin{proof}
 For set $P$ of $k$ points in general position, $|L(P)| = {k \choose 2}$. It follows that, for each $\ell \in L(P)$, $|\ell \cap (\mathcal L_P \setminus \ell)|\leq |L(P)|= {k \choose 2}\leq \frac{q}{2}$. The last inequality uses the assumption that $k \leq \sqrt q$.
 Therefore,
 $$|\mathcal L_P| \geq \sum_{\ell \in L(P)} (q - |\ell \cap (\mathcal L_P \setminus \ell)|) \geq  \frac{q}{2}\binom{k}{2}.$$
 The upper bound is trivial.
 
 \end{proof}
 
 The following lemma is the key result of this section.
 \begin{lemma} \label{lem:smallkey} For any $k \leq \sqrt{q}$, we have
 \[
 \prod_{i=2}^{k-1}\left (1 - \frac{i^2}{q}\right) \leq \P(k \text{ points in $\mathbb F_q^2$ form an arc}) \leq \prod_{i=1}^{k-2}\left(1 - \frac{i^2}{4q}\right)
 \]

 \end{lemma}
 \begin{proof}
We prove this result by induction on $k$. For the base step $k=2$, two points form an arc with probability $1$, and we have empty products on both sides of the desired inequality, so the base step is proved.


To perform the inductive step, notice that 
$$\P(k \text{ points form an arc}) = \P(k-1 \text{ points form an arc})\cdot \P(k \text{'th point is not collinear with two others}).$$
The first probability is bounded using the inductive hypothesis. The second can be bounded using Lemma \ref{lem:easy}.

 Let $P$ be an arc of cardinality $k-1$. The set of `bad points' (meaning points which are collinear with two points of $P$) are precisely the elements of $\mathcal L_P$. Lemma \ref{lem:easy} then gives
$$\P(\text{a point is bad}) = \frac{|\mathcal L_P| - (k-1)}{q^2 - (k-1)} \geq \frac{ \frac{q}{2}\binom{k-1}{2} - k}{q^2 } \geq \frac{(k-2)^2}{4q}.$$
On the other hand, also using Lemma \ref{lem:easy},
\[
\P(\text{a point is bad}) \leq \frac{q \binom{k-1}{2}}{q^2-(k-1)} \leq \frac{2 \binom{k-1}{2}}{q} \leq \frac{(k-1)^2}{q} .
\]
Therefore, 
$$1-  \frac{(k-1)^2}{q}\leq \P(\text{a point is good}) \leq 1 - \frac{(k-2)^2}{4q}.$$
We conclude that
$$\P(k \text{ points form an arc}) \leq \prod_{i=1}^{k-3}\left(1 - \frac{i^2}{4q}\right) \left(1 - \frac{(k-2)^2}{4q}\right) =\prod_{i=1}^{k-2}\left(1 - \frac{i^2}{4q}\right) $$
and
$$\P(k \text{ points form an arc}) \geq \prod_{i=2}^{k-2}\left(1 - \frac{i^2}{q}\right) \left(1 - \frac{(k-1)^2}{q}\right) =\prod_{i=2}^{k-1}\left(1 - \frac{i^2}{q}\right), $$
as required. \end{proof}
\begin{proof}[Proof of Theorem \ref{thm:smallt}]

In order to prove Theorem \ref{thm:smallt}, the final task is to find a convenient approximation for the quantities in the statement of Lemma \ref{lem:smallkey}. All we are doing here is giving an approximation which is perhaps more easily digestible for the reader. There may be some loss in this last step, and so we emphasise that a more accurate statement is given by Lemma \ref{lem:smallkey}.

We use the approximation $1 - x \leq e^{-x}$, which holds for all $x$. This gives
\begin{align*}\P(k \text{ points form an arc}) \leq \prod_{i=1}^{k-2}\left(1 - \frac{i^2}{4q}\right)\leq  \prod_{i=1}^{k-2}e^{\frac{-ci^2}{q}} \leq e^{\frac{-c'k^3}{q}}
\end{align*}
where $c'$ is some positive absolute constant. Therefore
$$A(q,k) = \P(k\text{ points form an arc}){q^2 \choose k} \leq e^{\frac{-c'k^3}{q}} {q^2 \choose k} ,$$
as required.


 
 A similar argument is used to bound $A(q,k)$ from below. By Lemma \ref{lem:smallkey},
 $$A(q,k) = \P(k\text{ points form an arc}){q^2 \choose k} \geq \prod_{i=2}^{k-1}\left(1 - \frac{i^2}{q}\right) {q^2 \choose k}.$$
 Assuming the condition $k \leq \frac{q^{1/2}}{1 + \delta}$ for some $\delta > 0$, we can ensure that $\frac{i^2}{q} < \frac{1}{(1+ \delta)^2}$. There then exists some constant $D = D(\delta) > 0$ such that the approximation $e^{-Dx} \leq 1-x $ is valid for all $ 0\leq x \leq \frac{1}{(1 + \delta)^2}$. This implies that
 $$A(q,k) \geq \prod_{i=2}^{k-1} e^{-D\frac{i^2}{q}} {q^2 \choose k} \geq e^{-C\frac{k^3}{q}}{q^2 \choose k},$$
 for some constant $C(\delta)>0$.
 \end{proof}
 
 \section*{Acknowledgements}
 
 The authors were partially supported by the Austrian Science Fund FWF Project P 30405-N32. We are very grateful to Peter Allen, Nathan Kaplan and Cosmin Pohoata for helpful discussions and also for pointing out several useful references. We also thank Krishna Kaipa for clarifying some details concerning the connection between arcs and MDS codes.

\end{document}